\documentclass[11pt]{nyjm}

\usepackage{pdfsync}

\usepackage{amsmath,amsthm,amssymb,enumerate}
\usepackage{latexsym}
\usepackage{amscd}
\usepackage{amssymb}
\usepackage{enumitem}
\usepackage{mathtools}

\numberwithin{equation}{section}

\theoremstyle{plain}  
\newtheorem{thm}[equation]{Theorem}
\newtheorem{prop}[equation]{Proposition}

\newtheorem{lemma}[equation]{Lemma}

\theoremstyle{definition}  
\newtheorem{defn}[equation]{Definition}

\newtheorem{remark}[equation]{Remark}

\newtheorem{induction}[equation]{Induction}

\newcommand{\ra}{\rightarrow}

\newcommand{\lra}{\longrightarrow}

\renewcommand{\O}{\mathcal O}

\newcommand{\zt}{{\mathbf Z}_2}

\newcommand{\cA}{\mathcal{A}}

 \newcommand{\iso} {\cong}

\begin{document}


\title
{The cohomology of the connective spectra for {K}-theory revisited}

\author{Donald M. Davis}
\address{Department of Mathematics, Lehigh University\\Bethlehem, PA 18015, USA}
\email{dmd1@lehigh.edu}
\author{W. Stephen Wilson}
\address{Department of Mathematics, Johns Hopkins University\\Baltimore, MD 01220, USA}
\email{wwilson3@jhu.edu}

\thanks {Without Maple software, the Induction step \ref{intermed} would never have been found.}
\thanks {{We are grateful for the referee's very careful reading and many 
useful suggestions for improved exposition.}}

\begin{abstract} 
The stable mod 2 cohomologies of the spectra for connective real and complex
K-theories are well known and easy to work with.  However, the known bases
are in terms of the anti-automorphism of Milnor basis
elements.  We offer simple bases in terms of admissible sequences
of Steenrod operations that come from the Adem relations.  In particular,
{a basis for $H^*(bu)$ is given by those $Sq^I$ with $I$ admissible and no 
$Sq^1$ or $Sq^{2^n+1}$ appearing for $n > 0$.  }
\end{abstract}

\maketitle


\section{Introduction}\label{intro}

Our goal is to give simple 
bases for the mod 2 cohomologies, $H^*(bo)$ and $H^*(bu)$, 
for the connective real and complex K-theory spectra respectively.

Let $I=(i_1,i_2,\ldots,i_k)$.  
We let $Sq^I = Sq^{i_1} Sq^{i_2} \cdots Sq^{i_k}$ be a composition of 
Steenrod squares.
We have the length of $I${, given by} $\ell(I)=k$, and
the degree of $I${, given by} $|I|=|Sq^I|=\sum i_s$.
We say $I$ is {\em admissible} if $i_{s} \ge 2i_{s+1}$ for all $s$.
For $I$ admissible, we have the excess, $e(I)=i_1-i_2-\cdots -i_k$.
The admissible $Sq^I$ form the Serre-Cartan basis
for the mod 2 Steenrod algebra, $\cA$ {(\cite{SerreEM,Cartan})}.
Let $\cA_1$ 
be the sub-algebra generated by $Sq^1$
and $Sq^2$.  Let $E_1$ be the sub-algebra generated by $Q_0 = Sq^1$ and 
$Q_1 = Sq^1 Sq^2 + Sq^2 Sq^1$.

Let $\zt$ be the integers mod 2.
It has been known for a long time (\cite{Stong}) that 
$H^*(bo) = \cA\otimes_{\cA_1}\zt = \cA/\!\!/\cA_1$ and 
$H^*(bu) = \cA\otimes_{E_1}\zt = \cA/\!\!/E_1$. 
The usual basis for $H^*(bo)$ involves applying the anti-automorphism to
Milnor basis (\cite{MilA}) elements $Sq(R)$ with $R=(4r_1,2r_2,r_3,\ldots)$.  
One can also  extract an exotic basis for $H^*(bo)$ from
\cite{Dena} that is probably related to the spaces in the Omega spectrum
that the elements are created on.

We can now state our main theorem:


\begin{thm} \label{main}
{A basis for $H^*(bo)$ is given by all $Sq^I$ with $I$ admissible, no $i_s = 2^n+1$
for $n \ge 1 $ and $i_k \ge 4$. 
The case $H^*(bu)$ is the same except that $i_k \ge 2$.}
\end{thm}
Along the way we needed some things about the Steenrod algebra
that may be of independent interest.

\begin{defn}
Let $T_b \subset \cA$ be the span of all admissible $Sq^I$ with $i_1 \le b$.
\end{defn}

\begin{prop}
\label{basic}
$Sq^a\, T_b \subset T_{n}$ where\\
\[
n=\begin{cases}
a&\text{if\ \  }a\ge {2}b\\
2b-1&\text{if\ \  }2b>a\ge b\\
a+b&\text{if\ \  }b>a>0.
\end{cases}
\]
\end{prop}

\begin{remark}
Note that it is always true that $Sq^a T_b \subset T_{a+b}$.
The way we eliminate the $Sq^{2^n+1}$ is as follows.  We consider
$J=(2^{n+1}+1,i_0,\ldots,i_k)$ admissible. If $Sq^J$ is
non-zero in  
$\cA/\!\!/\cA_1$,
we can write it as a sum of $Sq^K$ with $K$ admissible and $k_1 \le 2^{n+1}$.
Although we don't need it in this paper, we also show that $e(J) > e(K)$
for every such $K$.
\end{remark}

The subalgebra $\cA_n$ of $\cA$ is generated by $Sq^1, Sq^2, Sq^4, \cdots, Sq^{2^n}$.
As usual, let $\alpha(n)$ be the number of ones in the binary expansion of $n$.
{We had a brief hope that a basis for $\cA/\!\!/\cA_n$ would be given by
$Sq^I$ with $I$ admissible, $i_k \ge 2^{n+1}$ and with no $i_s$ with $\alpha(i_s-1)\le n $.
Unfortunately it was false already in degree 49 for $\cA_2$.  
The anti-automorphism of the Milnor element $Sq(8,4,2,1)$ is non-zero in degree 49.
However, a short calculation shows that the
suggested conjecture has no elements in degree 49.
}

One observation survived:

\begin{prop}
\label{An}
In $\cA/\!\!/\cA_n$, if $\alpha(m) \le n$, then $Sq^{m+1} \in T_m$.
\end{prop}

{
Our initial interest was in $tmf$ with $H^*(tmf)= 
\cA/\!\!/\cA_2$.
However, it was clear that not only was nothing known here, but
the same held true for $H^*(bo)$.  Calculations led to the conjecture
and eventually the theorem.  A conjecture for
$\cA/\!\!/\cA_2$ still eludes us.
}

We first prove the results about the Steenrod algebra.  
Then we apply these results to prove Theorem \ref{main}.

\section{Results on the Steenrod algebra}

We will make constant use of the Thom spectrum, $MO$, for the unoriented cobordism case.
From \cite{Thom} we know that $H^*(MO)$ is free over $\cA$ and that one copy of $\cA$
sits on the Thom class $U \in H^0(MO)$.
We need the Stiefel-Whitney (S-W) classes, $w_i \in H^i(BO)$, and the Thom isomorphism
$H^*(BO) \iso H^*(MO)$ that takes $w_i$ to $Uw_i$.  We need the connection
between the S-W classes and the Steenrod algebra {(\cite{WuRel})} given by
$Sq^n U = U w_n$ and
$$
Sq^i (w_j) = {\sum}_{t=0}^i \binom{j+t-i-1}{t} w_{i-t} w_{j+t}.
$$
Keep in mind that $Sq^n w_n = w_n^2$ and $Sq^i w_n = 0$ when $i > n$.

The cohomology, $H^*(BO)$, is a polynomial algebra on the S-W classes, \cite{MilSta}.
We put an order on the monomials.  We have $M < M'$ if the degree of $M'$ is
greater than that of $M$.  Next, if they have the same degree, the one with the
largest $w_n$ is greater.  If they have the same largest $w_n$, we go to the
next largest and so on.  To use Thom's examples from his paper:  $w_4 < w_4 w_1^2
< w_4 w_2 w_1 < w_4 w_3$.

\begin{lemma}[Thom, in the proof of II.8, \cite{Thom}]
\label{thom}
For $I=(i_1,i_2,\ldots,i_k)$ admissible, in $H^*(MO)$,
$$
Sq^I(U) = U ( w_{i_1} w_{i_2}\cdots w_{i_k} + \Delta)
$$
where $\Delta$ is a sum of monomials of lower order.
\end{lemma}

\begin{remark}
{The filtration is not a filtration of $\cA$-modules, but Thom's}
result allows us to distinguish between admissible $Sq^I$ using
the S-W classes in the Thom spectrum.
Because it was 1954, Thom worked in the stable range of $MO(n)$ where
his Thom class was $w_n$.
\end{remark}

\begin{proof}[Proof of Proposition \ref{basic}.]
It is enough to consider the case when $i_1 = b$.
When $a \ge 2b$, there is nothing to prove because $Sq^a Sq^I$
is already admissible.
When $2b > a \ge b$, this is no longer the case. 
If $Sq^a Sq^I$ is written in terms of admissible $Sq^J$, we
need to determine what the maximum possibility is for $j_1$.
We look at
\begin{eqnarray*}
Sq^a Sq^I(U) &=& Sq^a \Big( U ( w_{b} w_{i_2}\cdots w_{i_k} + \Delta)\Big) \\
&=& 
\sum_{j=0}^a Sq^{a-j}(U) Sq^j
( w_{b} w_{i_2}\cdots w_{i_k} + \Delta) \\
&=&
\sum_{j=0}^a Uw_{a-j} Sq^j ( w_{b} w_{i_2}\cdots w_{i_k} + \Delta).
\end{eqnarray*}
{Since $a-j < 2b$, }
 the largest possible new
S-W class is given by $Sq^j(w_b)$, but the largest this can be is $Sq^{b-1} (w_b)
= w_{2b-1}$ plus other terms with products.  Similarly, if $a=2b-1$, we
could get $w_{2b-1}$ when $j=0$ in the formula.  Not only is $n=2b-1$
the largest possible, but it is realized.

Using the same formula when $b > a > 0$, the largest possible $w_n$
is when $Sq^a w_b$ includes $w_{a+b}$ and that is only realized when $\binom{b-1}{a}=1$ (mod 2).
This concludes the proof.
\end{proof}

It is time to introduce one of our key tools, the Adem relations {(\cite{AdemRel})}:
$$
Sq^a Sq^b = \sum_{i}^{[a/2]} \binom{b-1-i}{a-2i} Sq^{a+b-i} Sq^i.
$$
These apply when $a < 2b$, that is,
when $Sq^a Sq^b$ is not admissible.  The resulting terms are admissible.
The sum is from the maximum of $0$ or $a-b+1$.

\begin{proof}[Proof of Proposition \ref{An}]
We induct on $m$.  Let $m=2^{k_1}+\cdots+2^{k_\ell}$ with $k_1 >  \cdots > k_\ell$
and $\ell \le n$.  Let $s = \min \{i : k_i > k_{i+1}+1\}$.  If no such $s$, let
$s = \ell$.  If $s = \ell$  and $k_\ell = 0$, then $m+1 = 2^\ell$, and we are
done since $Sq^{2^\ell} = 0$ in
$\cA/\!\!/\cA_n$.
Otherwise, write $m=2a+b$ with
\[
a=2^{k_1-1} + \cdots + 2^{k_s-1} \text{ and } 
b=2^{k_{s+1}} + \cdots + 2^{k_\ell}. 
\]
Note that if $s=\ell$, then $b=0$ and $m$ is even, having already done the odd case.
Then $Sq^{a+b+1} \in T_{a+b}$ by the induction hypothesis.
Proposition \ref{basic} says $Sq^a T_{a+b} \subset T_{2a+b}$, so
we have $Sq^a Sq^{a+b+1} \in T_{2a+b}$.
Since $\binom{a+b}{a} = 1$, ({This is because the binary expansion of
$a+b$ includes the binary expansion of $a$ in this case.}  We are always working mod 2), 
the Adem relation gives
\[
Sq^a Sq^{a+b+1} = \sum \binom{a+b-i}{a-2i} Sq^{2a+b+1-i} Sq^i
=Sq^{2a+b+1} + \Delta.
\]
Here $\Delta \in T_{2a+b}$ and so is the left hand side, so
$Sq^{m+1}=Sq^{2a+b+1} \in T_{2a+b}=T_m$.
\end{proof}

We now consider the obvious homomorphism from the span of admissible monomials
described in Theorem \ref{main} into
$\cA/\!\!/\cA_1$.
In Section \ref{bu}, we do the deduction for
$\cA/\!\!/E_1$.

\section{Injectivity}

We wish to show that the $Sq^I$ of Theorem \ref{main} are linearly
independent.
This follows directly from Lemma \ref{thom} once the background is set up.
For that we need the polynomial algebra from \cite{EThomas}
\[
H^*(BSpin) = P[w_i] \qquad i \ge 4  \qquad i \ne 2^n+1.
\]
We give a quick proof of this because it involves techniques we need anyway.

Mod decomposables, we have the following easily verified formulas:
\[
Sq^{2^k} (w_{2^k+1}) \equiv   w_{2^{k+1}+1}
\qquad
Sq^{(2^n,2^{n-1},\ldots,4,2,1)} (w_2) \equiv 
w_{2^{n+1}+1}.
\]
The second follows immediately from the first.  Because $w_2 = 0 \in H^*(BSpin)$
by definition, any Steenrod operations on it are zero as well.  
The formula tells us that 
$w_{2^{n+1}+1}$ is decomposable in $H^*(BSpin)$.  Most are non-trivial, but
we do have $w_3 = w_5 = w_9 = 0$.

We know $H^*(BSO) = P[w_i]$ with $i > 1$.
We have a fibration $BSpin \ra BSO \ra K_2 = K(\zt,2)$, where the last map is
given by $w_2$.
Note that $H^*(K_2)$ is a polynomial algebra on the
$Sq^{(2^n,2^{n-1},\ldots,4,2,1)} (\iota_2)$.
The above computation shows the map $H^*(K_2) \lra H^*(BSO)$
is an injection giving us a short exact sequence of Hopf algebras
$ H^*(K_2) \lra H^*(BSO) \lra H^*(BSpin)$ from the Eilenberg-Moore
(or Serre) spectral sequence.   
{The collapse of the EM-s.s. is because we are working with Hopf
algebras so the injection makes $H^*(BSO)$ free over $H^*(K_2)$.}
This gives $H^*(BSpin) $ and the
decomposability of the $w_{2^n+1}$.

We are also going to look at the Thom spectrum, $MSpin$.  
Let $U$ be the Thom class in $H^0(MSpin)$.
The reason we are looking at the Thom spectrum is because as a module over
the Steenrod algebra, $H^*(MSpin)$ is a sum of cyclic modules and the
module generated by $U$ is precisely
$\cA/\!\!/\cA_1$, \cite{ABP}.

\begin{proof}[Proof of injectivity for Theorem \ref{main}.]
If there were a relation in
$\cA/\!\!/\cA_1$
among the admissible $Sq^I$ with no $i_s = 2^n+1$ and $i_k \ge 4$, Lemma \ref{thom}
would imply a similar relations among the S-W classes in $H^*(MSpin)$.  But because
we are not using the $w_{2^n+1}$, these are linearly independent.
\end{proof}

\section{Surjectivity}

All that is left to do with our Theorem \ref{main} is to show that any
admissible $Sq^I$ with some $i_s = 2^n+1$ can be written in terms of
admissible $Sq^J$ with no $i_s =  2^t+1$.

We specialize Proposition \ref{basic} to $Sq^{2^n+j}T_{2^n} \subset T_{2^{n+1}}$
when $2^n > j \ge 0$.  Proposition \ref{basic} actually tells us $T_{2^{n+1}-1}$ but
we don't need that little extra bit.

\begin{lemma}
\label{stringk}
In   $\cA/\!\!/\cA_1$,
if $J=(2^{n+1}+1,i_0,\ldots,i_k)$ is admissible, then $Sq^J \in T_{2^{n+1}}$,
that is, $T_{2^{n+1}+1} \subset T_{2^{n+1}}$.
\end{lemma}

\begin{proof}[Proof of Theorem \ref{main} for $H^*(bo)$ from Lemma \ref{stringk}.]
{
If we have an $I$ admissible with some $i_s = 2^n+1$ with $n \ge 1$, 
we want to show that $Sq^I$ can be replaced without this $i_s=2^n+1$.
Write $I=LJ$ where $J$ is the shortest possible  as in Lemma \ref{stringk}.
Lemma \ref{stringk} tells us that 
$Sq^J$ can be written in terms of $Sq^K$ admissible with $k_1 < 2^{n+1}+1$. When
this sum replaces $Sq^J$ in $Sq^I$, $LK$ is still admissible.  
}
By induction, 
we do not have to worry
about smaller $J$ like this showing up.
Since $I$ is finite, this process of replacement is also finite.
We have shown that every $Sq^I$, $I$ admissible, can be replaced with one of the
desired form, and we have shown that the $Sq^I$ of this form are linearly independent.
This concludes the proof of Theorem \ref{main} from Lemma \ref{stringk}.
\end{proof}

\begin{proof}[Proof of Lemma \ref{stringk}.]
When $J=(2^{n+1}+1)$, we can use the 
$\cA/\!\!/\cA_1$
case
of Proposition \ref{An}.

To start our induction on $k$, we need the $J=(2^{n+1}+1,i_0)$ case.
We begin with $Sq^{2^n+1} = \Delta \in T_{2^n}$ from Proposition \ref{An}
and apply $Sq^{2^n+i_0}$. 
For $J$ to be admissible, we have $i_0 \le 2^n$.  
From Proposition \ref{basic},
$Sq^{2^n+i_0} T_{2^n} \subset T_{2^{n+1}}$ so $Sq^{2^n+i_0}Sq^{2^n+1} \in T_{2^{n+1}}$
since $Sq^{2^n+i_0}\Delta \in T_{2^{n+1}}$.  
All we need now is:
\begin{eqnarray*}
Sq^{2^n+i_0} Sq^{2^n+1} &=& \sum_{s \ge i_0} 
\binom{2^n-s}{2^n+i_0-2s} Sq^{2^{n+1} +1 +i_0 -s} Sq^s\\
&=& Sq^{2^{n+1} +1} Sq^{i_0} + \sum_{s > i_0} 
\binom{2^n-s}{2^n+i_0-2s} Sq^{2^{n+1} +1 +i_0 -s} Sq^s.
\end{eqnarray*}
The terms in the sum are all also in $T_{2^{n+1}}$ so the same is true for
$ Sq^{2^{n+1} +1} Sq^{i_0}$.

The following induction proves our Lemma \ref{stringk} because the two terms
in $T_{2^{n+1}}$ force the third term to be there as well.
The induction is started above as it can be rephrased in 
the format of
our induction below as the $k=0$ case.

\begin{induction} In   $\cA/\!\!/\cA_1$,
\label{intermed}
with
$(2^{n+1}+1,i_0,\ldots,i_k)$ admissible,
$$
Sq^{2^n+i_0} Sq^{2^{n-1} +i_1} \cdots Sq^{2^{n-k}+i_k} Sq^{2^{n-k}+1} \in T_{2^{n+1}}
$$
and is equal in
$\cA/\!\!/\cA_1$
 to
$$
Sq^{2^{n+1}+1}Sq^{i_0}Sq^{i_1}\cdots Sq^{i_k} + \Delta_{n+1}
\text{ with } \Delta_{n+1} \in T_{2^{n+1}}.
$$
\end{induction}

\begin{proof}[Proof of our Induction \ref{intermed}]

By induction on $k$, we can write
$$
Sq^{2^{n-1} +i_1} \cdots Sq^{2^{n-k}+i_k} Sq^{2^{n-k}+1} 
\in T_{2^n}
$$
and it is equal to
$$
Sq^{2^{n}+1}Sq^{i_1} \cdots Sq^{i_k}
+
\Delta_n
\text{ with } \Delta_n \in T_{2^n}.
$$

Now we take $Sq^{2^n+i_0}$ times everything.
Since $Sq^{2^n +i_0} T_{2^n} \subset T_{2^{n+1}}$, we have
$Sq^{2^n+i_0}\Delta_n = \Delta_{n+1} \in T_{2^{n+1}}$ and
$$
Sq^{2^n+i_0} Sq^{2^{n-1} +i_1} \cdots Sq^{2^{n-k}+i_k} Sq^{2^{n-k}+1} 
\in T_{2^{n+1}}
$$
and is equal in 
$\cA/\!\!/\cA_1$
to
$$
Sq^{2^n+i_0}Sq^{2^{n}+1}Sq^{i_1}\cdots Sq^{i_k} + \Delta_{n+1}.
$$
{So, the term
\[
Sq^{2^n+i_0} Sq^{2^{n}+1}
Sq^{i_1} \cdots Sq^{i_k} 
\]
is also in $T_{2^{n+1}}$. It is equal to}
\begin{eqnarray*}
&&\\
&=&\Big( \sum_{s \ge i_0} \binom{2^n -s}{2^n+i_0 -2s} Sq^{2^{n+1}+1+i_0 -s}Sq^s\Big)
Sq^{i_1} \cdots Sq^{i_k} \\
&&\\
&= &
Sq^{2^{n+1}+1}Sq^{i_0} 
Sq^{i_1} \cdots Sq^{i_k} \\
&&\\
&&+
\Big( \sum_{s > i_0} \binom{2^n -s}{2^n+i_0 -2s} Sq^{2^{n+1}+1+i_0 -s}Sq^s\Big)
Sq^{i_1} \cdots Sq^{i_k}.
\end{eqnarray*}
Since $s > i_0$, the elements in the sum are admissible and in $T_{2^{n+1}}$.
{They can now be incorporated into $\Delta_{n+1}$.
We are left with $Sq^J =
Sq^{2^{n+1}+1}Sq^{i_0} Sq^{i_1} \cdots Sq^{i_k} $ from Lemma \ref{stringk}
which is therefore also in $T_{2^{n+1}}$.
}
\end{proof}

Lemma \ref{stringk} follows. 
\end{proof}

Although we don't need this next Lemma, it is interesting
in its own right.
Let $E_r$ be spanned by all $Sq^I$, $I$ admissible, $e(I) \le r$.
Let $K(\zt,r) = K_r$ be the Eilenberg-MacLane space with $\iota_r \in H^r(K_r)$ the
fundamental class.
The significance of excess is that the $Sq^I \iota_r$ with $Sq^I \in E_r$ are 
linearly independent in $H^*(K_r)$
and $Sq^I \iota_r = 0$ for $e(I) > r$.

\begin{lemma}
\label{excess}
In   $\cA/\!\!/\cA_1$,
if $J=(2^{n+1}+1,i_0,\ldots,i_k)$ is admissible, then $Sq^J \in E_{e(J)-2}$.
\end{lemma}

\begin{proof}
In   $\cA/\!\!/\cA_1$, write $Sq^J = \sum Sq^K$ with $K$ admissible and,
from Lemma \ref{stringk}, $k_1 \le 2^{n+1}$.
If $Sq^J = 0$, there is nothing to prove.  We have 
$|J|=|K|$.    
For $I$ admissible,
recall $|I|= i_1 +\cdots {+} i_k$ and $e(I)=i_1-i_2 {-} \cdots -i_k$.
We can connect them with $2i_1 -|I| = e(I)$. Now
\[
e(K) = 2k_1 -|K| = 2k_1 - |J| \le 2^{n+2} -|J| = 2(2^{n+1}+1) -|J| -2 = e(J) -2.
\]
\end{proof}

\section{$H^*(bu)$}
\label{bu}

We give a quick derivation of $H^*(bu)$ from $H^*(bo)$.
We have the standard fibration
\[
bo \lra bu \lra \Sigma^2 bo.
\]
This gives a short exact sequence of $\cA$ modules.  The map from
$H^*(\Sigma^2 bo)$ takes $1$ to $Sq^2$ and is injective so must hit
all $Sq^I Sq^2$ with $Sq^I$ a basis for $H^*(bo)$.  The surjection
$H^*(bu) \lra H^*(bo)$ must hit the $Sq^I$ for a basis for $H^*(bo)$.
This is the stated answer for $H^*(bu)$ in Theorem \ref{main}.


\end{document}